\documentclass[12pt,a4paper,twoside]{article}
\usepackage{amsmath,amsthm,amssymb,amsfonts,amsbsy}

\usepackage{color, colortbl,xcolor} 
\usepackage{float} 
\usepackage[margin=2.5cm]{geometry} 
\usepackage{cite}
\usepackage{authblk} 
\usepackage{setspace}
\usepackage{lipsum}
\usepackage{mathrsfs}
\theoremstyle{definition}

\theoremstyle{plain}
\newtheorem{prob}[subsection]{Problem}

\newtheorem{lem}[subsection]{Lemma}
\newtheorem{rmk}[subsection]{Remark}
\newtheorem{prop}[subsection]{Proposition}
\newtheorem{cor}[subsection]{Corollary}

\newcommand{\beq}{\begin{eqnarray}}
\newcommand{\eeq}{\end{eqnarray}}

\newcommand{\beqs}{\begin{eqnarray*}}
\newcommand{\eeqs}{\end{eqnarray*}}

\allowdisplaybreaks

\usepackage{mathtools}

\title{\bf Two Irregularity Measures Possessing High Discriminatory Ability}

\author{
   \large  Akbar Ali${}^{\dag}$, Tam\'{a}s R\'{e}ti${}^{\ddag}$
}

\affil{  \normalsize \it
    $^\dag${University of Management and Technology, Sialkot 51310, Pakistan}
     \\ {\tt akbarali.maths@gmail.com}
}

\affil{ \normalsize
    { \it ${}^{\ddag}$\'{O}buda University, B\'{e}csi\'{u}t, 96/B, H-1034 Budapest, Hungary}
    \\ {\tt reti.tamas@bgk.uni-obuda.hu}
}

\begin{document}

\maketitle

\begin{abstract}

An $n$-vertex graph whose degree set consists of exactly $n-1$ elements is called antiregular graph. Such type of graphs are usually considered opposite to the regular graphs. An irregularity measure ($IM$) of a connected graph $G$ is a non-negative graph invariant satisfying the property:  $IM(G) = 0$ if and only if $G$ is regular. The total irregularity of a graph $G$, denoted by $irr_t(G)$, is defined as $irr_t(G)= \sum_{\{u,v\} \subseteq V(G)} |d_u - d_v|$ where $V(G)$ is the vertex set of $G$ and $d_u$, $d_v$ denote the degrees of the vertices $u$, $v$, respectively. Antiregular graphs are the most nonregular graphs according to the irregularity measure $irr_t$\,; however, various non-antiregular graphs are also the most nonregular graphs with respect to this irregularity measure. In this note, two new irregularity measures having high discriminatory ability are devised. Only antiregular graphs are the most nonregular graphs according to the proposed measures.\\[2mm]

\end{abstract}
%
%
%
%
%
%
%
%
\section[Introduction]{Introduction}

Graphs considered in this article are simple, connected and finite, unless stated otherwise. Sets of vertices and edges of a graph $G$ will be denoted by $V(G)$ and $E(G)$, respectively. Degree of a vertex $u$ and the edge connecting the vertices $u,v\in V(G)$ will be denoted by $d_u$ and $uv$, respectively. By an $n$-vertex graph, we mean a graph with $n$ vertices. An $(n,m)$-graph is an $n$-vertex graph with $m$ edges. The graph theoretical terminology, not defined here, can be found in some standard books of graph theory, like \cite{Bondy-08}.

The degree set of a graph $G$ is denoted \cite{6} by $\mathscr{D}(G)$ and is defined as the set of all different vertex degrees of $G$. A graph whose degree set consists of only one element is called regular graph. Sometimes (for example, see \cite{Alavi-87}), the term ``irregular graphs'' is used for those graphs which are not regular, while sometimes (for example, see \cite{6}), the same term is used for a totally different purpose. Hence, in order to avoid confusion, and by following the references \cite{Stevanovic-04,vanDam-98,Cioaba-07}, we use the term ``nonregular graphs'' instead of ``irregular graphs'' for the graphs which are not regular.

A graph having maximum degree less than 5 is known as a molecular graph in chemical graph theory. Molecular graphs of annulenes, cycloalkanes and fullerenes are the examples of regular molecular graphs. The vast majority of molecular graphs is nonregular; some are more nonregular than others.

An irregularity measure ($IM$) of a connected graph $G$ is a non-negative graph invariant satisfying the property:  $IM(G) = 0$ if and only if $G$ is regular.
If $IM(G) > IM(H)$ then we say that $G$ is more nonregular than $H$ according to the considered irregularity measure $IM$.
Irregularity measures may play an important role in network theory \cite{Criado-IJCM-14,Estrada-PRE-10,Estrada-ACS-10,Snijders-81a,Snijders-81b,Lawrence-95} as well as in chemistry, particularly in the QSPR (quantitative structure-property relationship) and QSAR (quantitative structure-activity relationship) studies \cite{Gutman-JCIM-05,Reti-MATCH-18}.

Historically, the Gini index (some detail about this index is given in Section \ref{sec-3}), appeared implicitly in \cite{Gini-1912},
can be considered as one of the first irregularity measures. However, this index was intended to be used for a completely different purpose \cite{Bendel-89,Sen-73}. For $m\ge1$, the Gini index for an $(n,m)$-graph $G$, denoted by $\zeta(G)$, can be defined as follows
\[
\zeta(G)=\frac{1}{2mn} \sum_{\{u,v\} \subseteq V(G)} |d_u - d_v|\,.
\]
Here, it needs to be mentioned that $\zeta(G)=\frac{irr_t(G)}{2mn}$ where $irr_t$ is a recently introduced irregularity measure, namely the total irregularity \cite{Abdo-14}.

We may say that Collatz and Sinogowits \cite{Collatz-57} introduced explicitly the first irregularity measure, which is defined, for an $(n,m)$-graph $G$, as
$$
CS(G)= \lambda_1 - \frac{2m}{n}\,,
$$
where $\lambda_1$ is the greatest eigenvalue of the adjacency matrix of $G$. For $n\ge3$, Estrada \cite{Estrada-PRE-10} devised the following irregularity measure, under the name ``normalized heterogeneity index'', within the study of network heterogeneity:
\[
\rho(G)= \frac{n-2\cdot R(G)}{n-2\sqrt{n-1}}\,,
\]
where $R(G)$ is the Randi\'c index \cite{r3,Li08} of the $n$-vertex graph $G$. Actually, dozens of irregularity measures exist in literature and various new ones can be easily defined. In Table 1, those existing irregularity measures (together with their definitions and some relevant references) are given which will be discussed in this paper. Further detail about the existing irregularity measures can be found in the surveys \cite{Gutman-16,Ananias-de-Oliveira-13}, papers \cite{Reti-MATCH-18,Boaventura-Netto-15,Elphick-14,Milovanovic-15,Milovanovic-16,Reti-18,Boaventura-Netto-18} and in the references listed therein.

It is well-known fact that there does not exist any $n$-vertex graph whose all degrees are different for $n>1$. An $n$-vertex graph whose degree set consists of exactly $n-1$ elements is called the antiregular graph \cite{34} as well as the quasi-perfect graph \cite{33}, half-complete graph \cite{Fishburn}, maximally nonregular graph \cite{Zykov} and pairlone graph \cite{Salehi}; it seems that ``antiregular graphs'' is a generally accepted term for referring such kind of graphs \cite{35,37,38} (also see \cite{36} for some basic properties of the antiregular graphs), so we use this term in the remaining part of this paper. It is known \cite{33} that for every integer $n\ge2$ there is a unique antiregular $n$-vertex graph $A_n$ (and a unique disconnected antiregular $n$-vertex graph, which is actually the complement of $A_n$). Following the references \cite{Abdo-14,Boaventura-Netto-18,Boaventura-Netto-15}, we take antiregular graphs as the graphs opposite to the regular graphs.



\vskip 18mm

\begin{center}
\centerline{{\bf Table 1.} Some existing irregularity measures considered in this paper.}\label{tab2.0}

\vskip 3mm

\setstretch{1.7}
\begin{tabular}{ll} \hline
Name of irregularity measure & Definition for an $(n,m)$-graph $G$ \\ \hline
Gini index \cite{Gini-1912}  & $\zeta(G)=\frac{1}{2mn} \sum_{\{u,v\} \subseteq V(G)} |d_u - d_v|$  \\
Collatz-Sinogowitz index \cite{Collatz-57}  & $CS(G)= \lambda_1 - \frac{2m}{n}$  \\
Degree variance \cite{Bell,Haviland-06,Snijders-81a,Snijders-81b}  & $Var(G)= \frac{1}{n}\sum_{v\in V(G)} \left(d_v - \frac{2m}{n}\right)^2$ \\
Discrepancy  \cite{Haviland-06,Lawrence-95} & $Disc(G)= \frac{1}{n} \sum_{v\in V(G)} |d_v - \frac{2m}{n}|$ \\
Albertson index  \cite{Alb} & $A(G)= \sum_{uv\in E(G)} |d_u - d_v|$ \\
Degree deviation \cite{Nikiforov-06}  & $S(G)= n\cdot Disc(G)$ \\
Normalized heterogeneity index \cite{Estrada-PRE-10}  & $\rho(G)= \frac{n-2\cdot R(G)}{n-2\sqrt{n-1}}$ \\
Total irregularity \cite{Abdo-14}  & $irr_t(G)= 2mn \cdot \zeta(G)$ \\
Sigma index  \cite{Furtula-15,Gutman-18} & $\sigma(G)= \sum_{uv\in E(G)} (d_u - d_v)^2$ \\
\hline
\end{tabular}
\end{center}

\vskip 3mm

The following problem was posed in \cite{Reti-Conf-19}.

\begin{prob}\label{prob-1}

Let $G$, $R$, $A_n$ be any $n$-vertex graph, an $n$-vertex regular graph, an $n$-vertex antiregular graph, respectively. Is there any irregularity measure $IM$ which satisfies the inequality
\begin{equation}\label{eq-1}
IM(R) \le IM(G) \le IM (A_n)
\end{equation}
with left equality if and only if $G\cong R$ and the right equality holds if and only if $G\cong A_n$?
\end{prob}

The main purpose of the present article is
to devise two new irregularity measures having high discriminatory ability as well as satisfying the constraints specified in Problem \ref{prob-1}.
The newly developed irregularity measures are compared with some well-known existing irregularity measures and it is noted that the proposed measures give better results in a certain way.

\section{Construction of Two Irregularity Measures Possessing High Discriminatory Performance}


Before defining the two new irregularity measures, we would like to note, from Table 2, that among those existing irregularity measures which are considered in this paper, only the graph invariant $|\mathscr{D}(G)|-1$ satisfies the constraints specified in Problem \ref{prob-1}. However, according to Gutman \cite{Gutman-16} ``In the case of molecular graphs, the invariant $|\mathscr{D}(G)|-1$ should be applied with due caution, or -- better -- not applied at all.
Because, for the graphs depicted in Figure \ref{fig0}, it holds that $|\mathscr{D}(H_1)|=|\mathscr{D}(H_2)|$; but, intuitively, one would expect
that $H_2$ is much more nonregular than $H_1$''. Also, we observe that the total irregularity $irr_t$ satisfies \eqref{eq-1} and the extremal graphs for the left inequality of \eqref{eq-1} are same as mentioned in Problem \ref{prob-1}. However, there exist graphs different from $A_n$ for which the right equality sign in \eqref{eq-1} holds. Consequently, we need to define some new irregularity measures satisfying the conditions mentioned in Problem \ref{prob-1}.




\vskip 6mm

\begin{center}
\centerline{{\bf Table 2.} Some existing irregularity measures of the four graphs, shown in Figure \ref{fig1}.}\label{tab2.1}

\vskip 3mm

\begin{tabular}{ccccccccccccc} \hline
Graph & $m$ & $irr_t$ & $|\mathscr{D}|-1$ & $CS$ & $A$ & $\sigma$ & $Var$ & $S$ & $\zeta$ & $\rho$\\ \hline
 $G_1$& 9& 26& $4$ & 0.404& 16 &40 &1.667 &6.000  & 0.241 & 0.304\\
 $G_2$& 7& 26& $3$ & 0.481& 18 &56 &1.889 &6.667  & 0.310 & 0.522\\
 $G_3$& 8& 26& $3$ & 0.435& 20 &56 &1.889 &7.333  & 0.271 & 0.419\\
 $G_4$& 8& 26& $2$ & 0.510& 14 &44 &1.889 &6.667  & 0.271 & 0.433\\ \hline
\end{tabular}
\end{center}

\vskip 3mm

\begin{figure}[h]
 \centering
    \includegraphics[width=0.40\textwidth]{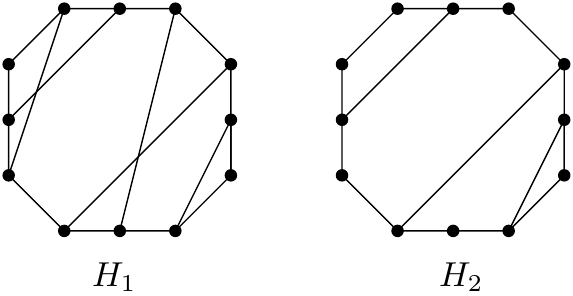}
    \caption{Two nonregular graphs with the same degree set.}
\label{fig0}
\end{figure}

\begin{figure}[h]
 \centering
    \includegraphics[width=0.40\textwidth]{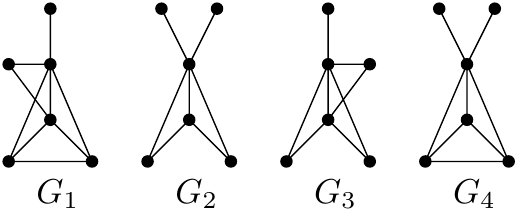}
    \caption{Four 6-vertex nonregular graphs with the same total irregularity.}
\label{fig1}
\end{figure}

From the computed irregularity measures given in Table 2, the following conclusions can be drawn for the graphs $G_i$ $(i=1,2,3,4)$:

\begin{itemize}
  \item For all graphs, the total irregularity index is same, that is $irr_t=26$, and hence we may say that $irr_t$ has a low discriminatory ability for the considered graphs.
  \item	Among the investigated irregularity measures, there are six measures ($CS$, $\sigma$, $Var$, $S$, $\zeta$ and $\rho$) having a minimum value for the antiregular graph $G_1$ and hence we conclude that this graph is less nonregular than each of the other three graphs according to these six irregularity measures.
  \item In addition to $irr_t$, the irregularity measures $\sigma$, $Var$, $S$, $\zeta$
  have only a limited discriminatory power for the graphs under consideration.
  Surprisingly, $Var(G_1)=1.667$, while $Var(G_2)= Var(G_3)= Var(G_4)=1.889$.
  \item For the majority of the considered irregularity measures, the right inequality in \eqref{eq-1} does not hold.

\end{itemize}

In what follows, it is assumed, unless stated otherwise, that $G$ is a graph of order at least 3 with the vertex set $V(G)=\{v_1,v_2,\cdots,v_n\}$ and degree sequence $(d_1,d_2,\cdots,d_n)$ such that $d_1\ge d_2\ge\cdots\ge d_n$ where $d_i = d_{v_i}$ for $i=1,2,\cdots,n$. We define an $n\times n$ matrix $\mathbf{B}(G)$ having entries
\[
b_{i,j}=
\begin{cases}
d_i - d_j & \text{if $i<j$,}\\
0 & \text{if $i=j$,}\\
d_j - d_i & \text{if $i>j$.}
\end{cases}
\]  							
The matrix $\mathbf{B}(G)$ is referred as the \textit{degree-difference matrix} of the graph $G$. Clearly, the matrix $\mathbf{B}(G)$ is a symmetric matrix. One can construct several different versions of the degree-difference matrix. For example, consider the $n\times n$ matrix $\mathbf{B^{(1)}}(G)$ of the graph $G$ whose components are defined by $b_{i,j}^{(1)}= d_i - d_j$.
Clearly, $\mathbf{B^{(1)}}(G)$ is an antisymmetric matrix. Another possible version of the degree-difference matrix is the matrix $\mathbf{B^{(2)}}(G)$ whose entries are defined as
$b_{i,j}^{(2)}= (d_i - d_j)^2$.

For a non-negative integer $k$, let $N_k(G)$ be the number of those upper diagonal entries of the matrix $\mathbf{B}(G)$ which are equal to $k$. In other words, $N_k(G)$ is the number of those pairs of vertices $(v_i, v_j)\in V(G)\times V(G)$ which satisfy $d_i - d_j = k$ for $i < j$.
Clearly, it holds that
\begin{equation}\label{eq-r1}
\sum_{k\ge0} N_k(G)=\frac{n(n-1)}{2}\,.
\end{equation}
Moreover, it is easy to see that
\[
irr_t(G)= \sum_{k\ge1} k N_k(G) 
\]
which is, also, equal to $\sum_{i<j} b_{i,j}$, that is, the sum of the upper diagonal entries of the matric $\mathbf{B}(G)$. 

We consider the 6-vertex nonregular graphs depicted in Figure \ref{fig1}. Because all of them have the same total irregularity (equal to 26), they cannot be distinguished in terms of their irregularity using the total irregularity. Efficient discrimination between graphs with equal total irregularity can be performed by constructing novel graph invariants that are highly sensitive to the structural differences in such graphs. The simplest such structure-sensitive invariant having an improved discriminatory power is the graph invariant $N_0(G)$. For the graphs shown in Figure \ref{fig1}, one obtains $N_0(G_j)=j$ for $j=1,2,3,4$. By means of $N_0(G)$, various irregularity measures can be generated; some of them seem to be efficient for the structural discrimination (ranking) of graphs with identical total irregularity. Here, we define the following two such irregularity measures
\[
IRA(G)=\frac{n(n-1)}{2}\cdot \frac{1}{N_0(G)} - 1 \quad \text{and} \quad  IRB(G)= 1-\frac{2}{n(n-1)}\cdot N_0(G) \,.
\]
Because of \eqref{eq-r1}, the formulas of the irregularity measures $IRA(G)$ and $IRA(G)$ can be rewritten as
\[
IRA(G)= \frac{1}{N_0(G)} \cdot \sum_{k\ge1} N_k(G)  \quad \text{and} \quad IRB(G)= \frac{2}{n(n-1)} \cdot \sum_{k\ge1} N_k(G)\,.
\]
Here, we note that $N_0(H_1)=46$ and $N_0(H_2)=30$, and hence using the irregularity measures $IRA$ and $IRB$, we remark that $H_2$ is more nonregular than $H_1$, as noted in the first paragraph of this section.

Now, for the 6-vertex graphs depicted in Figure \ref{fig1} having the same $irr_t$ value, we compute the newly defined irregularity measures $IRA$ and $IRB$; these are given in Table 3.
We note that the calculated values of $IRA$ and $IRB$ for these four graphs are all different and that the antiregular graph $G_1$ has the maximal values of the irregularity measures $IRA$ and $IRB$ among the considered graphs, which indicate that the measures $IRA$ and $IRB$ have a high discriminatory ability as well as these measures may satisfy the constraints given in Problem \ref{prob-1}, which is actually true due to Proposition \ref{prop-a1}. Consequently, we conclude that the newly developed irregularity measures $IRA$ and $IRB$ are somehow better, in a certain way, than the existing irregularity measures given in Table 1.

\vskip 9mm

\begin{center}
\centerline{{\bf Table 3.} The irregularity measures $IRA$ and $IRB$ of the graphs, shown in Figure \ref{fig1}.}\label{tab2.1}

\vskip 3mm

\begin{tabular}{cccccc} \hline
Graph & $m$ & $irr_t$ & $N_0$ & $IRA$ & $IRB$ \\ \hline
 $G_1$& 9& 26& $1$ &14.00   &0.933\\
 $G_2$& 7& 26& $2$ &6.50    &0.867\\
 $G_3$& 8& 26& $3$ &4.00    &0.800\\
 $G_4$& 8& 26& $4$ &2.75    &0.733\\ \hline
\end{tabular}
\end{center}

\vskip 3mm

\begin{lem}\label{obs-1}
It holds that $N_0(G)\ge1$ with equality if and only if $G$ is an antiregular graph.

\end{lem}

\begin{proof}
The desired result follows from the fact that every non-trivial graph contains at least two vertices of same degree.
\end{proof}

From the definitions of the irregularity measures $IRA$ and $IRB$, and from Lemma \ref{obs-1}, the next result follows.

\begin{prop}\label{prop-a1}
It holds that
\begin{equation}\label{eq-a1}
0\le IRA(G) \le \frac{n(n-1)}{2} - 1
\end{equation}
and
\begin{equation}\label{eq-a2}
0\le IRB(G) \le 1- \frac{2}{n(n-1)}\,.
\end{equation}
The left equality sign in either of Inequalities \eqref{eq-a1}, \eqref{eq-a2} holds if and only if $G$ is regular, while the right equality sign in either of Inequalities \eqref{eq-a1}, \eqref{eq-a2} holds if and only if $G$ is antiregular.

\end{prop}


Since $1- \frac{2}{n(n-1)}\rightarrow 1$ when $n\rightarrow \infty$, Proposition \ref{prop-a1} ensures that
the value of the irregularity measure $IRB(G)$ lies between 0 and 1.

Next, we compare the measures $IRA$ and $IRB$ with the total irregularity $irr_t$, which is a modified version of the Albertson index $A$. The following three facts can be considered as the main reasons for introducing $irr_t$ (the purpose of adding these three facts is that the irregularity measures $IRA$ and $IRB$ have all those advantages which the measure $irr_t$ has, and in addition, these two newly proposed measures also obeyed all the constraints given in Problem \ref{prob-1}, however $irr_t$ does not obey some of these constraints)\,:\\

\noindent
\textbf{Fact 1.}	The measure $irr_t$ can be calculated from the degree sequence of a graph $G$, while to calculate the Albertson index, one needs all the vertices' adjacency information for $G$;\\

\noindent
\textbf{Fact 2.}	Graphs with the same degree sequence have the same $irr_t$  value, while there exist some graphs, say $G$ and $G'$, with equal degree sequences such that $A(G) \ne A(G')$;\\

\noindent
\textbf{Fact 3.}    Among all the $n$-vertex graphs, the graphs with maximal Albertson index
are bidegreed graphs belonging to the family of complete split graphs \cite{Abdo-14-Filomat},
while the graphs with maximal $irr_t$ value have large degree sets.\\

\noindent
We remark that if ``$irr_t$'' is replaced with either ``$IRA$'' or ``$IRB$'' in the statements of Facts 1, 2, and 3, then the resulting statements also hold. In addition, we note that the graph with $(n-1)$-element degree set is the only graph with the maximal $IRA$ value (as well as maximal $IRB$ value) among all the  $n$-vertex graphs; this is not always the case for the total irregularity ``$irr_t$''.

Since the irregularity measures $IRA$ and $IRB$ depend only on the graph invariant $N_0$\,,
finding mathematical properties of the invariant $N_0$ seems to be an interesting work.

\begin{rmk}
From Lemma \ref{obs-1}, it follows that if the inequality $N_0(G) > 1$ holds then $G$ is not an antiregular graph. Consequently, the graph invariant $N_0(G)$ classify the $n$-vertex graphs into disjoint subsets (representing disjoint equivalence classes).

\end{rmk}

A graph whose degree set consists of only two elements is called a bidegreed graph. By a bidegreed partition $(A,B)$ of a bidegreed graph $G$, we mean a partition of $V(G)$ such that $d_u\ne d_v$ for every $u\in A$ and for every $v\in B$.

\begin{lem}\label{prop-8}
If the $n$-vertex nonregular graph $G$ has the maximum degree $\Delta$, then
\[
N_0(G) \le \frac{n(n-1)}{2} - \Delta\,
\]
with equality if and only if $G$ is a bidegreed graph containing a unique vertex of degree $n-1$.

\end{lem}

\begin{proof}
We note that
\[
N_0(G)= \sum_{i=1}^{\Delta} \frac{n_i(n_i-1)}{2}\,,
\]
where $n_i$ is the number of vertices of degree $i$ in $G$.
Suppose that one of the $n_i$'s is $n-k$ for some fixed $k\in \{1,2,\cdots,n-1\}$.

If $k \le \Delta$, then it holds that
\begin{align}\label{eq:1}
N_0(G) &\le \frac{(n-k)(n-k-1)}{2} + \frac{k(k-1)}{2} = \frac{n(n-1)}{2} - k(n-1) +k(k-1) \nonumber \\
  & \le \frac{n(n-1)}{2} - k\Delta +k(k-1) \le \frac{n(n-1)}{2} - \Delta\,.
\end{align}
The last inequality holds if $k(k-1) \le \Delta (k-1)$, which certainly obeyed.
We note that the equality sign holds throughout in \eqref{eq:1} if and only if $G$ is a bidegreed graph with the bidegreed partition $(A,B)$ such that one of $|A|$, $|B|$ is $k$ and the other is $n-k$, $\Delta=n-1$ and either $k=1$ or $k=\Delta$; that is, if and only if $G$ is a bidegreed graph containing a unique vertex of degree $n-1$.

If $k > \Delta$, then we have
\[
N_0(G) \le  \frac{n(n-1)}{2} - k(n-k) \le \frac{n(n-1)}{2} - k < \frac{n(n-1)}{2} - \Delta\,.
\]
because of $n-k\ge1$.

\end{proof}

The following proposition is a direct consequence of Lemma \ref{prop-8}.

\begin{prop}\label{prop-8new1}
If an $n$-vertex nonregular graph $G$ has the maximum degree $\Delta$, then
\[
IRA(G) \ge \frac{2\Delta}{n(n-1) -2\Delta} \,
\]
and
\[
IRB(G) \ge  \frac{2\Delta}{n(n-1)}\,,
\]
with equality if and only if $G$ is a bidegreed graph containing a unique vertex of degree $n-1$.

\end{prop}

If $H_1$, $H_2$ are two bidegreed $n$-vertex graphs with $n_\Delta(H_1)=n_\Delta(H_2)$ or $n_\delta(H_1)=n_\Delta(H_2)$ then from the equation
\[
N_0(G)= \sum_{i=1}^{\Delta} \frac{n_i(n_i-1)}{2}\,,
\]
the next result follows.

\begin{prop}\label{prop-8new1aa}
If $H_1$, $H_2$ are two bidegreed $n$-vertex graphs with $n_\Delta(H_1)=n_\Delta(H_2)$ or $n_\delta(H_1)=n_\Delta(H_2)$ then
$
IRA(H_1) = IRA(H_2)
$
and
$
IRB(H_1) = IRB(H_2)
$.
\end{prop}

The following corollary is direct consequence of Proposition \ref{prop-8new1aa}.

\begin{cor}\label{cor-8new1aa}
If $H_1$, $H_2$ are two regular $n$-vertex graphs such that $H_1-e_1$, $H_2-e_2$ are connected then
$
IRA(H_1-e_1) = IRA(H_2-e_2)
$
and
$
IRB(H_1-e_1) = IRB(H_2-e_2)
$, where $e_1\in E(H_1)$ and $e_2\in E(H_2)$.
\end{cor}

Several existing irregularity measures have different values for the graphs $P_6$ (the 6-vertex path graph which is isomorphic to the graph obtained from the 6-vertex cycle (a regular graph) graph by removing an edge) and $K_6-e$ (the graph obtained from the 6-vertex complete graph by removing an edge); for example, the Albertson index, Collatz-Sinogowitz index, Gini index, etc. and hence according to these irregularity measures, one of the two graphs $P_6$, $K_6-e$, is more nonregular than the other one.
Contrary to this, one intuitively would expect that both the graphs $P_6$ and $K_6-e$ have same degree of irregularity or better to say that neither of these two graphs is more nonregular than the other one; the same conclusion implies from Corollary \ref{cor-8new1aa}.
This example demonstrates clearly that $IRA$ and $IRB$ indices quantify basically the structural heterogeneity of the graphs $P_6$ and $K_6-e$. Strictly speaking, $IRA$ and $IRB$ indices characterize (measure) the heterogeneity (inhomogeneity) difference between the vertex-degree distributions of the considered graphs.

\section{A New/Old Formulation of the Total Irregularity}\label{sec-3}

Let $Y=(y_1,y_2,\cdots,y_n)$ be a sequence of non-negative real numbers $y_i$, for which it holds that $y_1\ge y_2\ge \cdots \ge y_n$ and that $\mu(Y) = \frac{\sum_{i=1}^n y_i}{n}\ne 0$. The Gini index $\zeta$ (also known as the Gini coefficient), attributed to Gini \cite{Gini-1912}, for the sequence $Y$ can be written (see page 31 in \cite{Sen-73}) as
\[
\zeta(Y)=\frac{1}{2n^2\cdot \mu(Y)} \sum_{i=1}^n\sum_{j=1}^n |y_i-y_j|=1-\frac{1}{n^2\cdot \mu(Y)} \sum_{i=1}^n (2i-1)y_i\,.
\]
Since the average degree of an $(n,m)$-graph $G$ containing at least one edge is $2m/n$, the Gini index for $G$ can be defined as follows
\[
\zeta(G)=\frac{1}{4mn} \sum_{i=1}^n\sum_{j=1}^n |d_i-d_j|=1-\frac{1}{2mn} \sum_{i=1}^n (2i-1)d_i\,,
\]
where $V(G)=\{v_1,v_2,\cdots,v_n\}$, $d_1\ge d_2\ge\cdots\ge d_n$, and $d_i = d_{v_i}$ for $i=1,2,\cdots,n$. Clearly, $\zeta(G)\ge0$ with equality if and only if $G$ is regular, which means that the Gini index is also an irregularity measure. Here, it needs to be mentioned that the Gini index is bounded between 0 and 1; for example, see \cite{Bendel-89}. We note that the total irregularity of the graph $G$ can be rewritten as
\begin{equation*}\label{Eq-new-01}
irr_t(G)= 2mn \cdot \zeta(G) = \sum_{i=1}^n (n+1 -2i)d_i =2m(n+1) - 2 \sum_{i=1}^n id_i\,.
\end{equation*}
In order to find the $irr_t$ value of a graph $G$, one may prefer the formula $irr_t(G)= \sum_{i=1}^n (n+1 -2i)d_i$ due to its simplicity instead of $irr_t(G)=\frac{1}{2}\sum_{i=1}^n\sum_{j=1}^n |d_i-d_j|$.





\begin{thebibliography}{999}
\setlength{\itemsep}{0pt}

\bibitem{Abdo-14} H. Abdo, S. Brandt, D. Dimitrov, The total irregularity of a graph, \textit{Discrete Math. Theor. Comput. Sci.} \textbf{16} (2014) 201--206.


\bibitem{Abdo-14-Filomat} H. Abdo, N. Cohen, D. Dimitrov, Graphs with maximal irregularity, \textit{Filomat} \textbf{28} (2014) 1315--1322.


\bibitem{35} C. O. Aguilar, J. L. Lee, E. Piato, Spectral characterization of anti-regular graphs, \textit{Linear Algebra Appl.} \textbf{557} (2018) 84--104.

\bibitem{Alb} M. O. Albertson,  The irregularity of a graph, \textit{Ars Combin.} \textbf{46} (1997) 219--225.

\bibitem{Alavi-87} Y. Alavi, G. Chartrand, F. R. K. Chung, P. Erd\H{o}s, R. L. Graham, O. R. Oellermann, Highly irregular graphs, \textit{J. Graph Theory} \textbf{11} (1987) 235--249.

\bibitem{Ananias-de-Oliveira-13} J. Ananias de Oliveira, C. Silva Oliveira, C. Justel, N. M. Maia de Abreu, Measures of irregularity of graphs, \textit{Pesquisa Operacional} \textbf{33(3)} (2013) 383--398.

\bibitem{Bell} F. K. Bell, A note on the irregularity of graphs, \textit{Linear Algebra Appl.} \textbf{161} (1992) 45-–54.

\bibitem{Bendel-89} R. B. Bendel, S. S. Higgins, J. E. Teberg, D. A. Pyke, Comparison of skewness coefficient, coefficient of variation, and Gini coefficient as inequality measures within population, \textit{Oecologia} \textbf{78} (1989) 394--400.

\bibitem{33} M. Bezhad, G. Chartrand, No graph is perfect, \textit{Amer. Math. Monthly} \textbf{74} (1967) 962--963.

\bibitem{Boaventura-Netto-15} P. O. Boaventura-Netto,   Graph irregularity: discussion, graph extensions and new proposals, \textit{Revista de Matem\'atica: Teor\'ia y Aplicaciones} \textbf{22(2)} (2015) 293--310.

\bibitem{Boaventura-Netto-18} P. Boaventura-Netto, L. de Lima, G. Caporossi, Exhaustive and metaheuristic exploration of two new structural irregularity measures, preprint.

\bibitem{Bondy-08} J. A. Bondy, U. S. R. Murty, \textit{Graph Theory}, Springer, 2008.

\bibitem{6} G. Chartrand, P. Erd\H{o}s, O. Oellermann, How to define an irregular graph, \textit{College Math. J.} \textbf{19} (1988) 36--42.

\bibitem{Cioaba-07} S. M. Cioab\u{a}, D. A. Gregory, V. Nikiforov, Extreme eigenvalues of nonregular graphs, \textit{J. Combin. Theory Ser. B} \textbf{97} (2007)  483--486.

\bibitem{Collatz-57} L. Collatz, U. Sinogowitz, Spektren endlicher Grafen, \textit{Abh. Math. Sem. Univ. Hamburg} \textbf{21} (1957) 63--77.

\bibitem{Criado-IJCM-14} R. Criado, J. Flores, A. G. del Amo and M. Romance, Centralities of a network and its line graph: an analytical comparison by means of their irregularity, \textit{Int. J. Comput. Math.} \textbf{91} (2014) 304--314.

\bibitem{vanDam-98} E. R. V. Dam, Nonregular graphs with three eigenvalues, \textit{J. Combin. Theory Ser. B} \textbf{73} (1998) 101--118.

\bibitem{Elphick-14} C. Elphick, P. Wocjan, New measures of graph irregularity, \textit{El. J. Graph Theory Appl.} \textbf{2(1)} (2014) 52--65.


\bibitem{Estrada-PRE-10} E. Estrada, Quantifying network heterogeneity, \textit{Phys. Rev. E} \textbf{82} (2010) 066102.

\bibitem{Estrada-ACS-10} E. Estrada, Randi\'c index, irregularity and complex biomolecular networks, \textit{Acta Chim. Slov.} \textbf{57} (2010) 597--603.

\bibitem{Fishburn} P. C. Fishburn, Packing graphs with odd and even trees, \textit{J. Graph Theory} \textbf{7} (1983) 369-–383.

\bibitem{Furtula-15} B. Furtula, I. Gutman, \v{Z}. K. Vuki\'cevi\'c, G. Lekishvili, G. Popivoda, On an old/new degree–based topological index, \textit{Bull. Acad. Serbe Sci. Arts} (\textit{Cl. Sci. Math. Natur.}) \textbf{40} (2015) 19--31.

\bibitem{Gini-1912} C. Gini, Variabilit\`{a} e mutabilit\`{a}, Reprinted in: E. Pizetti, T. Salvemini (Eds.), \textit{Memorie di metodologica statistica}, Rome, Libreria Eredi Virgilio Veschi, 1912.

\bibitem{Gutman-16} I. Gutman, Irregularity of molecular graphs, \textit{Kragujevac J. Sci.} \textbf{38} (2016) 99--109.


\bibitem{Gutman-JCIM-05} I. Gutman, P. Hansen, H. M\'elot, Variable neighborhood search for extremal graphs.10. Comparison of irregularity indices for chemical trees, \textit{J. Chem. Inf. Model.} \textbf{45}(2) (2005) 222--230.

\bibitem{Gutman-18} I. Gutman, M. Togan, A. Yurttas, A.S. Cevik, I.N. Cangul, Inverse problem for sigma index, \textit{MATCH Commun. Math. Comput. Chem.} \textbf{79} (2018) 491--508.

\bibitem{Haviland-06} J. Haviland, On irregularity in graphs, \textit{Ars Combin.} \textbf{78} (2006) 283--288.

\bibitem{Lawrence-95} C. J. Lawrence, K. Tizzard, J. Haviland, Disease-spread and stochastic graphs, \textit{Proc. International Conference on Social Networks}, London  1995, pp. 143--150.

\bibitem{38} V. E. Levit, E. Mandrescu, On the independence polynomial of an antiregular graphs, \textit{Carpathian J. Math.} \textbf{28} (2012) 279--288.

\bibitem{Li08} X. Li, Y. Shi, A survey on the Randi\'{c} index, \textit{MATCH Commun. Math. Comput. Chem.} \textbf{59} (2008) 127--156.

\bibitem{34} R. Merris, Antiregular graphs are universal for trees, \textit{Univ. Beograd. Publ. Elektrotehn. Fak. Ser. Mat.} \textbf{14} (2003) 1--3.

\bibitem{Milovanovic-15} E. Milovanovi\'c, E. Glogi\'c, I. Milovanovi\'c, M. Cvjetkovi\'c, Irregularity measures of graphs, \textit{Sci. Publ. State Univ. Novi Pazar Ser. A: App. Math. Inform. Mechan.} \textbf{7(2)} (2015) 105--116.

\bibitem{Milovanovic-16} I. Milovanovi\'c, E. Milovanovi\'c, V. \'Ciri\'c, N. Jovanovi\'c, On some irregularity measures of graphs, \textit{Sci. Publ. State Univ. Novi Pazar Ser. A: App. Math. Inform. Mechan.} \textbf{8(1)} (2016) 21--34.

\bibitem{37} E. Munarini, Characteristics, admittance and matching polynomials of an antiregular graph, \textit{Appl. Anal. Discrete Math.} \textbf{3} (2009) 157--176.


\bibitem{36} L. Nebesk\'{y}, On connected graphs containing exactly two points of the same degree, \textit{\v{C}asopis P\v{e}st Mat.} \textbf{98} (1973) 305--306.

\bibitem{Nikiforov-06} V. Nikiforov, Eigenvalues and degree deviation in graphs, \textit{Linear Algebra Appl.} \textbf{414} (2006) 347--360.

\bibitem{r3} M. Randi\'{c}, On characterization of molecular branching, \textit{J. Am. Chem. Soc.} \textbf{97} (1975) 6609--6615.

\bibitem{Reti-18} T. R\'eti, Graph irregularity and a problem raised by Hong, \textit{Acta Polytech. Hung.} \textbf{15(6)} (2018) 27--43.

\bibitem{Reti-Conf-19} T. R\'eti, A. Ali, On the comparative study of nonregular networks, preprint.

\bibitem{Reti-MATCH-18} T. R\'eti, R. Sharafdini, \'A. Dr\'egelyi-Kiss, H. Haghbin, Graph irregularity indices used as molecular descriptors in QSPR studies, \textit{MATCH Commun. Math. Comput. Chem.} \textbf{79} (2018) 509--524.

\bibitem{Salehi} E. Salehi, On $P_3$-degree of graphs, \textit{J. Combin. Math. Combin. Comput.} \textbf{62} (2007) 45--51.

\bibitem{Sen-73} A. Sen,  \textit{On Economic Inequality}, Oxford University Press (Clarendon), New York, 1973.

\bibitem{Snijders-81a} T. A. B. Snijders, Maximum value and null moments of the degree variance, Department of Mathematics Report TW-229, University of Groningen, 1981.
\bibitem{Snijders-81b} T. A. B. Snijders, The degree variance: an index of graph heterogeneity, \textit{Social Networks} \textbf{3} (1981) 163--174.

\bibitem{Stevanovic-04} D. Stevanovi\'c, The largest eigenvalue of nonregular graphs, \textit{J. Combin. Theory Ser. B} \textbf{91} (2004) 143-–146.

\bibitem{Zykov} A. A. Zykov, Fundamentals of graph theory, BCS Associates, Moscow, 1990.






























\end{thebibliography}
\end{document}